\documentclass[a4paper,11pt,reqno]{amsart}

\date{\today}
\usepackage[latin1]{inputenc}

\usepackage{
amssymb,
tikz,
xcolor,
}
\usepackage{physics}
\usepackage[hidelinks]{hyperref}
\usepackage{bbm}
\usepackage{tikz-cd}    
\usepackage{geometry}
\newtheorem{thm}{Theorem}[section]
\newtheorem{lemma}[thm]{Lemma}
\newtheorem{proposition}[thm]{Proposition}

\theoremstyle{remark}
\newtheorem{remark}[thm]{Remark}

\newtheorem{defi}[thm]{Definition}

\newcommand{\R}{\mathbb{R}}
\newcommand{\C}{\mathbb{C}}
\newcommand{\Z}{\mathbb{Z}}

\newcommand{\cF}{\mathcal{F}}
\newcommand{\cD}{\operatorname{\mathcal{D}}}

\renewcommand\phi{\varphi}
\newcommand{\eps}{\varepsilon}

\renewcommand{\leq}{\leqslant}

\newcommand{\Del}{\operatorname{\Delta}}
\newcommand{\ep}{\varepsilon}

\renewcommand{\leq}{\leqslant}

\newcommand\cH{{\mathcal{H}}}

\newcommand\cB{{\mathcal{B}}}

\newcommand\cR{{\mathcal{R}}}

\newcommand\I{{\rm{i}}}

\newcommand{\be}{\begin{equation}}
\newcommand{\ee}{\end{equation}}

\newcommand{\Hm}[1]{\leavevmode{\marginpar{\tiny%
$\hbox to 0mm{\hspace*{-0.5mm}$\leftarrow$\hss}%
\vcenter{\vrule depth 0.1mm height 0.1mm width \the\marginparwidth}%
\hbox to
0mm{\hss$\rightarrow$\hspace*{-0.5mm}}$\\\relax\raggedright #1}}}

\title[Bifurcating standing waves in gapped honeycomb structures]{Bifurcating standing waves for effective equations in gapped honeycomb structures}

\author[W. Borrelli]{William Borrelli}
\address[W. Borrelli]{Centro De Giorgi, Scuola Normale Superiore, Piazza dei Cavalieri 3, I-56100 , Pisa, Italy.}
\email{william.borrelli@sns.it}

\author[R. Carlone]{Raffaele Carlone}
\address[R. Carlone]{Universit\`{a} ``Federico II'' di Napoli, Dipartimento di Matematica e Applicazioni ``R. Caccioppoli'', MSA, via Cinthia, I-80126, Napoli, Italy.}
\email{raffaele.carlone@unina.it}

\date{\today}

\begin{document}

\begin{abstract}
In this paper we deal with two-dimensional cubic Dirac equations appearing as effective model in gapped honeycomb structures. We give a formal derivation starting from cubic Schr\"odinger equations and prove the existence of standing waves bifurcating from one band-edge of the linear spectrum.
\end{abstract}

\maketitle
{\footnotesize

\medskip

}


\section{Introduction}
\label{sec-intro}

\subsection{Motivation and main results}

In this paper we deal with nonlinear massive Dirac equations of the form
\be\label{eq:nld}
(\cD+m\sigma_3-\omega)\psi=h(\psi)\psi  \qquad\mbox{on}\quad\R^{2}\,,
\ee
where $\omega\in(-m,m)$ is a frequency in the spectral gap of the Dirac operator $\cD+m\sigma_3$, with $m>0$ (see Section \ref{sec:preliminaries}).  

\noindent We consider the nonlinearity in \eqref{eq:nld} of the form
\be\label{eq:nonlin}
h(z) = \begin{pmatrix}
\beta_1 |z_1|^2 + 2\beta_2|z_2|^2 & 0 \\
0 & \beta_1|z_2|^2 + 2\beta_2 |z_1|^2
\end{pmatrix}\,,\qquad z\in\R^2\,,
\ee
with given parameters $\beta_1,\beta_2>0$.
\medskip

Equation \eqref{eq:nld} appears as an effective model of wave propagation in two-dimensional honeycomb structures. As proved in \cite{FWhoneycomb}, if $V\in C^\infty(\R^2,\R)$ is a potential having the symmetries of a honeycomb lattice, then the Schr\"odinger operator
\be\label{eq:hamiltonian}
H=-\Del+V(x)\,,\qquad x\in\R^2\,,
\ee
exhibits generically conical touching points in its dispersion bands called \emph{Dirac points}. The dynamics of wave packets spectrally concentrated around Dirac points, see \cite{FWwaves}, is thus effectively described by the \emph{massless} (i.e., $m=0$) Dirac operator. Adding a perturbation that breaks parity induces a mass term in the effective operator, as proved in \cite[Appendix]{FWhoneycomb}. 

\vspace{1cm}

An important model in nonlinear optics and in the description of macroscopic phenomena is given by  the \emph{nonlinear Schr\"odinger / Gross-Pitaevski equation} \cite{derivation,nonlinearoptics,boseeinstein}
\begin{equation}\label{eq:nls}
\I\partial_t u=H u+\vert u\vert^2u .
\end{equation}

This equation, in the approximation described before, leads (at least formally) to the effective cubic nonlinearity \eqref{eq:nonlin}. Indeed, as first computed in \cite{wavedirac}, the effective equation around Dirac points reads
\begin{equation}\label{eq:effective}
\left\{\begin{aligned}
    \partial_{t}\Xi_{1}+\overline{\lambda}(\partial_{x_{1}}+i\partial_{x_{2}})\Xi_{2} &=i(2\beta_{2}\vert\Xi_{1}\vert^{2}+\beta_{1}\vert\Xi_{2}\vert^{2})\Xi_{1} \\
     \partial_{t}\Xi_{2}+\lambda(\partial_{x_{1}}-i\partial_{x_{2}})\Xi_{1} &=i(\beta_{1}\vert\Xi_{1}\vert^{2}+2\beta_{2}\vert\Xi_{2}\vert^{2})\Xi_{2} 
\end{aligned}\right.\,,
\end{equation}
where the parameters $\lambda\in\C\setminus\{0\}$, $\beta_1,\beta_2>0$ depend on the potential $V$ in \eqref{eq:hamiltonian}. 

Setting $\Psi_1:=-\frac{\lambda}{\vert\lambda\vert}\Xi_2\,, \Psi_2:=\Xi_1$ and looking for stationary solutions
\[
\Psi(t,x)=\psi(x)\,,
\] we get the massless version of \eqref{eq:nld}, i.e. with $m=\omega=0$. As shown in Section \ref{sec:derivation}, adding a perturbation breaking the parity of the potential $V$ in \eqref{eq:hamiltonian} gives an additional mass term in the effective equation. This corresponds to a gap $(-m,m)$ in the linear spectrum so that we can consider stationary solutions at frequency $\omega\in(-m,m)$, leading to \eqref{eq:nld}.

In \cite{arbunichsparber} the validity of the effective cubic equation is studied. In Section \ref{sec:derivation} we give a formal derivation of the effective model \eqref{eq:nld} using a multiscale expansion.

Existence and qualitative properties of solutions to the massless version of \eqref{eq:nld} have been studied in \cite{massless,borrellifrank}. The massive case in \eqref{eq:nld} has been addressed in \cite{shooting,mult} for the special choice of parameters in \eqref{eq:nld}. 

In this paper we partly generalize those results dealing with arbitrary $\beta_1,\beta_2>0$ and proving the existence of stationary solutions bifurcating from one edge of the spectral gap of the operator $\cD+m\sigma_3$.
\smallskip

We remark that cubic Dirac equations in two dimensions are \emph{critical} for the Sobolev embedding. Such types of equations have been studied also in different contexts. We mention, for instance, problems from conformal spin geometry, for which we refer the reader to \cite{spinorialanalogue,nadineconformalinvariant,isobecritical,maalaoui} and references therein, and in the case of coupled systems involving the Dirac operator and critical nonlinearities related to supersymmetric models coupling gravity with fermions, see \cite{Borrelli-Maalaoui-JGA2020,Maalaoui-Martino-JDE19}. The main difficulty in studying those equations comes from the underlying conformal symmetry so that looking for stationary solutions by variational methods one has to deal with the induced loss of compactness, see \cite{shooting,mult}. This problem can be circumvented, for instance, using a bifurcation argument to find solutions to \eqref{eq:nld}, as done in this paper following \cite{Ounaies}. We mention that the same method has been recently used for nonlinear Dirac equations on star graphs \cite{Diracnoncompactgraphs}. 
\medskip

The results given in \cite{shooting,mult} correspond to the choice of parameters $\beta_1=2\beta_2$, so that one can assume $\beta_1=1, \beta_2=1/2$ by scaling.
In this paper we deal with general $\beta_1,\beta_2>0$, but this forces us to put restrictions on the frequency $\omega$ that will be close to to the band-edge at $m$.
More precisely, we focus on the existence of standing waves to \eqref{eq:nld} of symmetric form
\begin{equation}\label{eq:ansatz}
\psi(r,\theta)
=
\begin{pmatrix}
v(r)  \\
\I u(r) e^{\I \theta}
\end{pmatrix} \,,\qquad (r,\theta)\in(0,\infty)\times\mathbb{S}^1\,
\end{equation}
 $(r,\theta)$ being polar coordinates in $\R^2$, and $u,v$ real-valued functions. Notice that \eqref{eq:ansatz} is the two-dimensional analogue of the \emph{Soler/Wakano ansatz} \cite{cv,cuevas,es}.

\begin{thm}\label{thm:main}
Let $\eps:=m-\omega$. There exists $\eps_0>0$ such that for $\eps\in(0,\eps_0)$ equation \eqref{eq:nld} admits a solution $\psi_\eps$ of the form \eqref{eq:ansatz}, with
\[
u_\eps(r)=\eps(-f'(\sqrt{\eps}r)+e_1(\sqrt{\eps}r))\,,\qquad v_\eps(r)=\sqrt{\eps}(f(\sqrt{\eps}r)+e_2(\sqrt{\eps}r)) \,,\qquad r>0\,,
\]
where $\Vert e_j\Vert_{H^1(\R^2)}\leq C\eps$, $j=1,2$, and $f\in H^1(\R^2)$ is the positive ground state of the NLS
\[
-\Delta f-f^3+f=0\,,\qquad \mbox{on $\R^2$.}
\]
\end{thm}

\begin{remark}
Arguing as in Section \ref{sec:proof} one can deal with the regime $-m<\omega<0,\omega\to-m$. However in that case the limit equation \eqref{eq:nls} is replaced by the following \emph{defocusing} NLS
\be\label{eq:defocusing}
-\Delta U+U^3+U=0\,,\qquad \mbox{on $\R^2$,}
\ee
which has no non-trivial solution in $H^1(\R^2)$. This can be easily seen multiplying the equation by such a solution and integrating by parts.
\end{remark}


\section{The Dirac operator}\label{sec:preliminaries}
 The \emph{Dirac operator} is the constant coefficients first order differential operator defined in two dimensions as
\begin{equation}\label{eq:dirac}
\mathcal{D}_{m}=\mathcal{D}+m\sigma_{3}:=-\I\sigma\cdot\nabla+m\sigma_{3}
\end{equation}
The constant $m>0$ usually represents the mass of the particle described by the equation. We adopt the notation $\sigma\cdot\nabla:=\sigma_{1}\partial_{1}+\sigma_{2}\partial_{2}$ and the $\sigma_{k}$'s are the Pauli matrices
\begin{equation}\label{eq:pauli} \sigma_{1}:=\begin{pmatrix} 0 \quad& 1 \\ 1 \quad& 0 \end{pmatrix}\quad,\quad \sigma_{2}:=\begin{pmatrix} 0 \quad& -\I \\ \I \quad& 0 \end{pmatrix} \quad,\quad \sigma_{3}:=\begin{pmatrix} 1 \quad& 0 \\ 0 \quad& -1 \end{pmatrix}\,.\end{equation}

The operator $\mathcal{D}_{m}$ is a self-adjoint operator on $L^{2}(\mathbb{R}^{2},\mathbb{C}^{2})$, with domain $H^{1}(\mathbb{R}^{2},\mathbb{C}^{2})$ and form-domain $H^{1/2}(\mathbb{R}^{2},\mathbb{C}^{2})$. 

Passing to the Fourier domain $p=(p_{1},p_{2})$ the Dirac operator \eqref{eq:dirac} becomes the multiplication by the matrix 
$$ \widehat{\mathcal{D}}_{m}(p)=\begin{pmatrix}m \quad & p_{1}-\I p_{2}\\ p_{1}+\I p_{2}\quad &m \end{pmatrix}$$ 
and then the spectrum is easily found to be
\begin{equation}\label{eq:spectrum}
Spec(\mathcal{D}_{m})=(-\infty, -m]\cup[m, +\infty)
\end{equation}
The above mentioned results can be found, e.g.,  in \cite{diracthaller}.
\section{Formal derivation of the model}\label{sec:derivation}
In this section we give a formal derivation of equation \eqref{eq:nld} from the corresponding cubic Schr\"odinger equation with honeycomb potential, following the exposition given in \cite{tesi}. 
\smallskip

We consider a fixed triangular lattice $\Lambda:=\Z v_1\oplus \Z v_2$, where $v_1,v_2\in\R^2$ are two linearly independent vectors.

\subsection{Honeycomb Schr\"odinger operators}
Consider the Schr\"odinger operator
\be\label{eq:schrodinger}
H:=-\Delta+V(x),\qquad x\in\R^{2}\,.
\ee
\begin{defi}\label{honeycomb}
The function $V\in C^{\infty}(\R^{2})$ is called \emph{honeycomb potential}, see \cite{FWhoneycomb}, if there exists $x_{0}\in\R^{2}$ such that $\tilde{V}(x)=V(x-x_{0})$ has the following properties:
\begin{enumerate}
\item $\tilde{V}$ is periodic with respect to some triangular lattice $\Lambda$, that is, $\tilde{V}(x+v)=\tilde{V}(x)$, $\forall x\in\R^{2},\forall v\in\Lambda$; 
\item $\tilde{V}$ is even: $\tilde{V}(-x)=\tilde{V}(x)$, $\forall x\in\R^{2}$;
\item $\tilde{V}$ is invariant by $\frac{2\pi}{3}$ counteclockwise rotation:  $$\cR[\tilde{V}](x):=\tilde{V}(R^{*}x)=\tilde{V}(x)\,\quad \forall x\in\R^{2},$$ where $R$ is the corresponding rotation matrix:
\be\label{rotationmatrix} R=\begin{pmatrix} -\frac{1}{2} & \frac{\sqrt{3}}{2}\\ -\frac{\sqrt{3}}{2}&-\frac{1}{2}   \end{pmatrix}.\ee
\end{enumerate}
\end{defi}

\begin{remark}{(Some examples of honeycomb potentials \cite{FWhoneycomb})}
\begin{enumerate}
\item{\textbf{Atomic potentials}:} Let $\mathbb{H}=(A+\Lambda)\cup(B+\Lambda)$ be a hexagonal lattice, given by the superposition of two triangular lattices. Consider a radial function $V_{0}\in C^{\infty}(\R^{2})$ raplidly decaying at infinity (for instance, with polynomial rate) representing the potential generated by a nucleous located on a vertex of the lattice. The potential $$V(x)=\sum_{y\in\mathbb{H}}V_{0}(x-y) $$
is then given by the superposition of atomic potentials, and it is a honeycomb potential (Def.\ref{honeycomb}).
\item{\textbf{Optical lattices:}} The envelop $\psi$ of the electric field of a monochromatic beam propagating in a dielectric medium can be described by a Schr\"odinger equation. More precisely, denoting by $z$ the direction of propagation of the beam and assuming that the refraction index varies only in the transversal directions $(x,y)$, the function $\psi$ solve the following equation
\be
i\partial_{z}\psi=\left(-\Delta+V(x,y)\right)\psi\,.
\ee
In this case the honeycomb potential is generated using optical interference techniques \cite{photoniclattice}. A typical example, is the potential of the form
\be
V(x,y)\simeq V_{0}\left(\cos(k_{1}\cdot(x,y))+\cos(k_{1}\cdot(x,y))+\cos((k_{1}+k_{2})\cdot x) \right),\quad V_{0}\in\R, k_{1},k_{2}\in\R^{2}.
\ee
\end{enumerate}
\end{remark}
For any fixed $k\in\R^{2}$ consider the following eigenvalue problem with \emph{pseudo-periodic boundary conditions }(see \cite{FWhoneycomb} and \cite[Sec. XIII.16]{reedsimonIV}) :
\be\label{pseudoperiodic}
\begin{cases}
H\Phi(x;k)=\mu(k)\Phi(x;k),\qquad x\in\R^{2}\\
\Phi(x+v;k)=e^{ik\cdot v}\Phi(x;k),\qquad v\in\Lambda.
\end{cases}
\ee
\begin{remark}
The eigenfunctions $\Phi(x;k)$ are of class $C^{\infty}$ by elliptic regularity theory.
\end{remark}
Recall that, given a lattice, its \emph{(first )Brillouin zone} $\cB$ is defined as the fundamental cell of the dual lattice. In the case of a honeycomb lattice, both its fundamental cell and its Brillouin zone are hexagonal \cite{FWhoneycomb}. An important property of $\cB$ is that waves propagating in a periodic medium can be described in terms of \emph{Bloch functions}.
\smallskip

Given $k\in\cB$, the resolvent of $H(k)$ is compact and then the spectrum of the operator is real and purely discrete, accumulating at $+\infty$:
\be\label{spectrumpseudoperiodic}
\mu_{1}(k)\leq\mu_{2}(k)\leq ...\leq\mu_{j}(k)\leq...\uparrow+\infty.
\ee
Fixing $n\in\mathbb{N}$, one says that $k\mapsto\mu_{n}(k)$ is the $n$-th \emph{dispersion band} of the operator $H$ and call $n$-th \emph{Bloch wave} the function $\Phi_{n}(x,k)$. The spectrum may also have some gaps, and it can be obtained as union of the images of the dispersion bands of the operator
\be\label{spectrumschrodinger}
\operatorname{Spec}(H)=\bigcup_{n\in\mathbb{N}}\mu_{n}(\cB)\,.
\ee

Moreover, the Bloch waves constitute a complete systems, meaning that for all  $f\in L^{2}(\R^{2})$
\be\label{complete}
f(x)-\sum_{1\leq n\leq N}\int_{\cB}\langle\Phi_{n}(\cdot,k),f(\cdot)\rangle_{L^{2}(\R^{2})}\Phi_{n}(x;k)dk \longrightarrow 0
\ee
in $L^{2}(\R^{2})$, for $N\longrightarrow+\infty$ \cite{FWhoneycomb,reedsimonIV}. 

The Cauchy problem
\be\label{cauchyschrodinger}
\begin{cases}
i\partial_{t}u(t,x)=Hu(t,x),\qquad (t,x)\in\R\times\R^{2},\\
u(0,x)=u_{0}(x)\in L^{2}(\R^{2}),
\end{cases}
\ee
admits the solution
\be\label{propagator}
e^{-iH_{V}t}u_{0}=\sum_{n\in\mathbb{N}}\int_{\cB}e^{-i\mu_{n}(k)}\langle\Phi_{n}(\cdot,k),u_{0}(\cdot)\rangle_{L^{2}(\R^{2})}\Phi_{n}(x,k)dk.
\ee
As a consequence, it is evident that the dynamics (\ref{propagator}) are strongly influenced by the behavior of the band functions $\mu_{n}(\cdot)$, $n\in\mathbb{N}$. In particular, as showed in \cite{FWwaves}, there exist two bands $\mu_N,\mu_{N+1}$ that meet at conical points located at the vertices of $\cB$. That is, locally near such a point $K_*\in\cB$ there holds

\be\label{conicaltouching}
\begin{cases}
\mu_{N+1}(k)-\mu_{N+1}(K_*)=\vert\lambda\vert\vert k-K\vert\left(1+E_{+}(k-K) \right),\\
\mu_{n}(k)-\mu_N(K_*)=-\vert\lambda\vert\vert k-K\vert\left( 1+E_{-}(k-K)\right)
\end{cases}\,,\qquad \vert k-K\vert<\delta\,,\lambda\in\C,\lambda\neq0\,.
\ee
Here $E_{\pm}: U_{\delta}\rightarrow \R$, with $U_{\delta}:=\left\{y\in\R^{2} :\vert y\vert<\delta \right\}$, are Lipschitz functions  such that  $E_{\pm}(y)=O(\vert y\vert)$, for $\vert y\vert\rightarrow0$. This means that, to first order, the dispersion relation near $k=K_*$ is a cone. This corresponds to the dispersion relation of the two-dimensional Dirac operator \eqref{eq:dirac}, as it can be readily seen in the Fourier domain.
\medskip

Consider a wave packet $u_{0}(x)=u^{\varepsilon}_{0}(x)$ concentrated around a Dirac point $K_*$
\be\label{concentrated}
u^{\varepsilon}_{0}(x)=\sqrt{\varepsilon}(\psi_{0,1}(\varepsilon x)\Phi_{1}(x)+\psi_{0,2}(\varepsilon x)\Phi_{2}(x))
\ee
where $\Phi_{j}$, $j=1,2$, are the Bloch functions at $K_*$ and the functions $\psi_{0,j}$ are some (complex) amplitudes to be determined. Then the solution of the NLS \eqref{eq:nls}, with initial conditions $u_0^\epsilon$ is expected to evolve to leading order in $\varepsilon$ still as a modulation of Bloch functions,
\be\label{approximate}
u^{\varepsilon}(t,x)\underset{\epsilon\rightarrow0^{+}}{\sim}\sqrt{\varepsilon}\left(\psi_{1}(\varepsilon t,\varepsilon x)\Phi_{1}(x)+\psi_{2}(\varepsilon t,\varepsilon x)\Phi_{2}(x) + \mathcal{O}(\varepsilon)\right)\,,\quad t>0,x\in\R^2\,,
\ee
and the amplitudes $\psi_j$ solve the effective equation \eqref{eq:effective}.
\medskip

Given a Dirac point $K_*\in\cB$, let $\mu_*:=\mu_N(K_*)=\mu_{N+1}(K_*)$ be the frequency at which the conical crossing occurs. Consider then the NLS
\be\label{eq:gp}
(-\Delta+V-\mu_*)u=\vert u\vert^2u\,,\qquad \R^2\,.
\ee

As in \eqref{approximate}, one thus looks for solutions to \eqref{eq:gp} of the form
\be\label{eq:approximate}
u^{\varepsilon}(t,x)\underset{\epsilon\rightarrow0^{+}}{\sim}\sqrt{\varepsilon}e^{-t\mu_*}\left(\psi_{1}(\varepsilon x)\Phi_{1}(x)+\psi_{2}(\varepsilon x)\Phi_{2}(x) + \mathcal{O}(\varepsilon)\right)\,,\quad t>0,x\in\R^2\,.
\ee
\subsection{Derivation of the massless equation}
The aim of this subsection is to formally derive the effective Dirac equation for the amplitudes $\psi_{j}$ appearing in \eqref{eq:approximate} through a multiscale expansion (see e.g. \cite{arbunichsparber,ilanweinstein}). 

Since the coefficients $\psi_{j}(\eps x)$ and the Bloch functions $\Phi_{j}(x)$ vary on different scales, one can consider $x$ and $y:=\eps x, 0<\eps\ll1$, as independent variables. Moreover, we look for solution to \eqref{eq:gp} as formal power series in $\eps$, as follows
\be\label{multiscaleu}
u_{\eps}=\sqrt{\eps}U_{\eps}(x,y),\qquad U_{\eps}(x,y)=U_{0}(x,y)+\eps U_{1}(x,y)+\eps^{2}U_{2}(x,y)+...
\ee 
We moreover impose $K_*$-pseudoperiodicity with respect to $x$, i.e. 
\be\label{pseudoperiodicmultiscale}
U_{\eps}(x+v,y)=e^{-iK_*\cdot v}U_{\eps}(x,y),\qquad\forall v\in\Lambda, x,y\in\R^{2}.
\ee 
Similarly, we look for $\mu$ of the form
\be\label{multiscalemu}
\mu=\mu_{\eps}=\mu_{*}+\eps\mu_{1}+\eps^{2}\mu_{2}+...
\ee
Rewriting \eqref{eq:gp} in terms of $U_{\eps}$ and $\mu_{\eps}$ then gives
\be\label{multiscaleNLS}
\left(-\left(\nabla_{x}+\eps\nabla_{y} \right)^{2}+V(x)-\mu_{\eps}\right)U_{\eps}(x,y)=\eps\left\vert U_{\eps}(x,y) \right\vert^{2}U_{\eps}(x,y).
\ee
Plugging (\ref{multiscaleu},\ref{multiscalemu}) into (\ref{multiscaleNLS}) one finds a hierarchy of equations. 

At order $\mathcal{O}(\eps^{0})$ we obtain
\be\label{order0}
(-\Delta_{x}+V-\mu_{*})U_{0}=0.
\ee
Recall that $\ker_{L^{2}_{K_*}}(-\Delta+V-\mu_{*})=\operatorname{Span}\left\{\Phi_{1},\Phi_{2} \right\}$, and then by (\ref{pseudoperiodicmultiscale}) we have
\be\label{uzero}
U_{0}(x,y)=\psi_{1}(y)\Phi_{1}(x)+\psi_{2}(y)\Phi_{2}(x),
\ee
where the amplitudes are to be determined solving the next equation in the formal expansion. Here $L^2_{K_*}$ denotes square integrable functions satisfying the pseudo-periodicity condition in \eqref{pseudoperiodic}.

The equation for $\mathcal{O}(\eps)$ terms reads
\be\label{uuno}
(-\Delta_{x}+V-\mu_{*})U_{1}=\left(2\nabla_{x}\cdot\nabla_{y}+\mu_{1} \right)U_{0}+\left\vert U_{0}\right\vert^{2}U_{0}.
\ee
By Fredholm alternative, solvability of the above equation requires its right hand side to be $L^{2}$-orthogonal to the kernel of $(-\Delta_{x}+V-\mu_{*})$. Then the functions $\psi_{j}$ are determined imposing orthogonality to the Bloch functions $\Phi_{k}$. For simplicity we deal with linear part and the cubic term in the right hand side of (\ref{uuno}) separately.  

The linear terms can be calculated using the following lemma from \cite{FWhoneycomb}
\begin{lemma}\label{products}
Let $\zeta=(\zeta_{1},\zeta_{2})\in\C^{2}$ be a vector. Then there exists $\lambda\in\C\setminus\{0\}$ such that we have
\begin{equation}\label{linearort}
\begin{split}
&\langle\Phi_{k},\zeta\cdot\nabla\Phi_{k}\rangle_{L^{2}(\Omega)}=0,\qquad k=1,2, \\& 2i\langle\Phi_{1},\zeta\cdot\nabla\Phi_{2}\rangle_{L^{2}(\Omega)}=\overline{2i\langle\Phi_{2},\zeta\cdot\nabla\Phi_{1}\rangle_{L^{2}(\Omega)}}=-\overline{\lambda}\left(\zeta_{1}+i\zeta_{2} \right),\\& 2i\langle\Phi_{2},\zeta\cdot\nabla\Phi_{1}\rangle_{L^{2}(\Omega)}=-\lambda(\zeta_{1}-i\zeta_{2})
\end{split}
\end{equation}
\end{lemma}
Notice that $\left(\nabla_{x}\cdot\nabla_{y}\right)U_{0}=\sum^{2}_{j=1}\nabla_{y}\psi_{j}\cdot\nabla_{x}\Phi_{j}$ and then applying Lemma \ref{products} with $\zeta=\nabla_{y}\Phi_{j}, j=1,2$ we get
\be\label{lineareffective}
\begin{split}
& 2i\langle\Phi_{1},\nabla_{y}\psi_{2}\cdot\nabla\Phi_{2}\rangle_{L^{2}(\Omega)}=\overline{2i\langle\Phi_{2},\nabla_{y}\psi_{2}\cdot\nabla\Phi_{1}\rangle_{L^{2}(\Omega)}}=-\overline{\lambda}\left(\partial_{y_{1}}+i\partial_{y_{2}} \right)\psi_{2},\\& 2i\langle\Phi_{2},\nabla_{y}\psi_{1}\cdot\nabla\Phi_{1}\rangle_{L^{2}(\Omega)}=-\lambda(\partial_{y_{1}}-i\partial_{y_{2}} )\psi_{1}
\end{split}
\ee
Thus we see that taking the $L^{2}(\Omega)$ scalar product of the linear part in the right hand side of (\ref{uuno}) with the Bloch functions $\Phi_{j}$ gives the linear part of \eqref{eq:effective}. We now want to show that the cubic nonlinearity in \eqref{eq:effective} is obtained calculating the same product for the cubic term in (\ref{uuno}). By symmetry taking this projection many terms vanish. The cubic term reads  
\be
\left\vert U_{0}\right\vert^{2}U_{0}=\sum_{1\leq j,k,l\leq2}\psi_{j}\psi_{k}\overline{\psi_{k}}\Phi_{j}\Phi_{k}\overline{\Phi_{l}}.
\ee 
Let us consider, for instance, the term $\psi_{1}\psi_{1}\overline{\psi_{2}}\Phi_{1}\Phi_{1}\overline{\Phi_{2}}$ and then project it onto $\Phi_{1}$. We compute
\be\label{cubicprojection}
\begin{split}
\langle \Phi_{1},\Phi_{1}\Phi_{1}\overline{\Phi_{2}}\rangle_{L^{2}(\Omega)}&=\int_{\Omega}\overline{\Phi_{1}(x)}\Phi_{1}(x)\Phi_{1}(x)\overline{\Phi_{2}(x)} dx \\&=^{x=R^{*}y}\int_{R\Omega}\overline{\Phi_{1}(R^{*}y)}\Phi_{1}(R^{*}y)\Phi_{1}(R^{*}y)\overline{\Phi_{2}(R^{*}y)} dy\\& \int_{R\Omega}\overline{\tau\Phi_{1}(y)}\tau\Phi_{1}(y)\tau\Phi_{1}(y)\overline{\overline{\tau}\Phi_{2}(y)} dy\\ & \tau^{2}\int_{\Omega}\overline{\Phi_{1}(x)}\Phi_{1}(x)\Phi_{1}(x)\overline{\Phi_{2}(x)} dx=\tau^{2}\langle \Phi_{1},\Phi_{1}\Phi_{1}\overline{\Phi_{2}}\rangle_{L^{2}(\Omega)}
\end{split}
\ee
where $R$ is the rotation matrix (\ref{rotationmatrix}), and we used that  $\cR\Phi_{1}=\tau\Phi_{1}$ and $\cR\Phi_{2}=\overline{\tau}\Phi_{2}$ with $\tau=\exp(2i\pi/3)$, see \cite{FWhoneycomb}. From (\ref{cubicprojection}) we get
$$(1-\tau^{2})\langle \Phi_{1},\Phi_{1}\Phi_{1}\overline{\Phi_{2}}\rangle_{L^{2}(\Omega)}=0, $$ and thus $$ \langle \Phi_{1},\Phi_{1}\Phi_{1}\overline{\Phi_{2}}\rangle_{L^{2}(\Omega)}=0.$$
Iterating this calculations one can check that 
\be\label{cubiceffective}
\left\{\begin{aligned}
      \langle\Phi_{1},\left\vert U_{0}\right\vert^{2}U_{0}\rangle_{L^{2}(\Omega)}&=(2\beta_{2}\vert\psi_{1}\vert^{2}+\beta_{1}\vert\psi_{2}\vert^{2})\psi_{1} \\
    \langle\Phi_{2},\left\vert U_{0}\right\vert^{2}U_{0}\rangle_{L^{2}(\Omega)} &=(\beta_{1}\vert\psi_{1}\vert^{2}+2\beta_{2}\vert\psi_{2}\vert^{2})\psi_{2} 
\end{aligned}\right.
\ee
thus recovering the cubic term in \eqref{eq:effective}, with
\be\label{betas}
\beta_{1}:=\int_{\Omega}\vert \Phi_{1}\vert^{4}dx=\int_{\Omega}\vert \Phi_{2}\vert^{4}dx,\qquad \beta_{2}:=\int_{\Omega}\vert\Phi_{1}\vert^{2}\vert\Phi_{2}\vert^{2}.
\ee
 It is then easy to see that the stationary version of \eqref{eq:effective} (i.e. $\partial_t\Xi_1=\partial_t\Xi_1=0$) appears as compatibility condition for the solvability of (\ref{uuno}), combining (\ref{lineareffective}, \ref{cubiceffective}) and taking $\mu_{1}=0$ in (\ref{uuno}).

\subsection{Derivation of the effective mass tem}
The same multiscale argument as in the previous Section allows to derive the mass term in \eqref{eq:effective}, induced by a suitable perturbation.

 As shown in \cite[Appendix]{FWhoneycomb}, breaking the $\mathcal{PT}$ symmetry lifts the conical degeneracy in the dispersion relation of a honeycomb Schr\"{o}dinger operator $\left(-\Delta+V \right)$ admitting Dirac points. Let us consider the following equation
\be\label{perturbedgp}
\left(-\Delta+V+\eps W-\mu_{*} \right)u=\vert u\vert^{2}u,
\ee
that is, we consider a potential perturbation of \eqref{eq:gp} where we add a linear term $W$ breaking parity. More precisely, we assume that $W$ is \emph{odd}
\be\label{odd}
W(-x)=-W(x),\qquad \forall x\in\R^{2}.
\ee
In this case, compared to the analysis in the previous Section, we get an additional term at order $\mathcal{O}(\eps)$ corresponding to the potential $\eps W$ in (\ref{perturbedgp}). Then we have to compute the projections
\be\label{wprojections}
\langle WU_{0},\Phi_{k}\rangle_{L^{2}(\Omega)}=\sum^{2}_{j=1}\psi_{j}\langle W\Phi_{j},\Phi_{k}\rangle_{L^{2}(\Omega)},\qquad k=1,2.
\ee
Recall that \be\label{phi12}\Phi_{2}(x)=\overline{\Phi_{1}(-x)},\ee and this relation allows us to compute
\be\label{12projection}
\begin{split}
\langle W\Phi_{2},\Phi_{1}\rangle_{L^{2}(\Omega)}&=\int_{\Omega}\left( W\Phi_{2}\right)(x)\overline{\Phi_{1}(x)}dx=\int_{\Omega} W(x)\overline{\Phi_{1}(-x)}\overline{\Phi_{1}(x)}dx\\&=^{y=-x}\int_{-\Omega} W(-y)\overline{\Phi_{1}(y)}\overline{\Phi_{1}(-y)}dy=-\int_{\Omega}\left( W\Phi_{2}\right)(y)\overline{\Phi_{1}(y)}dy\\&=-\langle W\Phi_{2},\Phi_{1}\rangle_{L^{2}(\Omega)},
\end{split}
\ee
where we have also used (\ref{odd}). We thus obtain
\be
\langle W\Phi_{2},\Phi_{1}\rangle_{L^{2}(\Omega)}=\langle W\Phi_{1},\Phi_{2}\rangle_{L^{2}(\Omega)}=0.
\ee
Moreover, arguing as in (\ref{12projection}) one easily finds
\be
\langle W\Phi_{1},\Phi_{1}\rangle_{L^{2}(\Omega)}=-\langle W\Phi_{2},\Phi_{2}\rangle_{L^{2}(\Omega)}.
\ee
and then
\be
\begin{split}
&\sum^{2}_{j=1}\psi_{j}\langle W\Phi_{j},\Phi_{1}\rangle_{L^{2}(\Omega)}=\psi_1\langle W\Phi_{1},\Phi_{1}\rangle_{L^{2}(\Omega)}\\& \sum^{2}_{j=1}\psi_{j}\langle W\Phi_{j},\Phi_{2}\rangle_{L^{2}(\Omega)}=-\psi_2\langle W\Phi_{1},\Phi_{1}\rangle_{L^{2}(\Omega)}.
\end{split}
\ee
Assuming that $m:=\langle W\Phi_{1},\Phi_{1}\rangle_{L^{2}(\Omega)}>0$, we obtain the mass term in \eqref{eq:effective}.

\section{Proof of the main result}\label{sec:proof}
In this section, we prove the existence of branches of bound states for \eqref{eq:nld} that bifurcate from the trivial solution at the positive band-edge of the spectrum of $\cD$. Those solutions are constructed from bound states of a suitable nonlinear Schr\"odinger equation \eqref{eq:nls}, which (after scaling) gives the asymptotic profile as $\omega\to m$.

\smallskip

We start by rewriting \eqref{eq:nld} componentwise. Setting $\psi=(\psi_1,\psi_2)^T$, equation \eqref{eq:nld} becomes the system
\be\label{eq:componenti}
\left\{
\begin{array}{l}
-i(\partial_1-i\partial_2)\psi_2 =(\beta_1 \vert \psi_2\vert^2+2\beta_2 \vert\psi_1\vert^2)\psi_1-(m-\omega)\psi_1\,\\[.3cm]
-i (\partial_1+i\partial_2)\psi_1=-(2\beta_2 \vert\psi_2\vert^2+\beta_1 \vert\psi_1\vert^2)\psi_2-(m+\omega)\psi_2\,,
\end{array}
\right.
\ee
that can be regarded as a functional equation of the form
\[
\cH(\psi_1,\psi_2)=0\,,
\]
where $\cH:X\times X\to L^2(\R^2,\C^2)$ is the map defined by
\[
\cH(\psi_1, \psi_2)=\begin{pmatrix}
-i(\partial_1-i\partial_2)\psi_2 -(\beta_1 \vert \psi_2\vert^2+2\beta_2 \vert\psi_1\vert^2)\psi_1+(m-\omega)\psi_1\\
-i (\partial_1+i\partial_2)\psi_1+(2\beta_2 \vert\psi_2\vert^2+\beta_1 \vert\psi_1\vert^2)\psi_2+(m+\omega)\psi_2
\end{pmatrix}\,,
\]
with 
\be\label{eq:space}
X:=H^1(\R^2,\C)\,.
\ee
In what follows, we shall consider the subspace
\be\label{eq:radialspace}
X_r\subset X\times X
\ee
given by functions of the form \eqref{eq:ansatz}. For simplicity they will be denoted by $(u,v)$, where those functions are the radial factors in \eqref{eq:ansatz}.

\subsection{Rescaling the equation}
\label{sec-scaling}
Plugging the ansatz \eqref{eq:ansatz} in \eqref{eq:componenti} leads to the follows system for the real valued functions $u,v$

\begin{equation}
\left\{
\label{eq:uv}
\begin{array}{l}
\displaystyle-u'+\frac{u}{r} =(\beta_1 u^2+2\beta_2 v^2)v-(m-\omega)v\,\\[.3cm]
v' =-(2\beta_2 u^2+\beta_1 v^2)u-(m+\omega)u\,.
\end{array}
\right.
\end{equation}

Now set $\eps:=(m-\omega)$, and consider the following rescaling 
\be\label{eq:scaleuv} 
u_\eps(r)=\eps f(\sqrt{\eps}r)\,,\qquad v_\eps(r)=\sqrt{\eps}g(\sqrt{\eps}r)\,,\qquad r>0\,,
\ee
so that, by \eqref{eq:uv}, after some straightforward computations we find the equations for $f(\rho),g(\rho)$:
\be
\label{eq:diracscaled}
 \left\{
 \begin{array}{l}
\displaystyle f'+\frac{f}{\rho}=(\eps\beta_1f^2+2\beta_2g^2)g-g\\[.3cm]
  \displaystyle g'= -(2\beta_2\eps^2f^2+\beta_1\eps g^2)f-(2m-\eps)f\,.
 \end{array}
 \right.
\ee
where we also used the fact that $\eps:=m-\omega$ and then $m+\omega=2m-\ep$. 

\begin{remark}
The branch point of the solutions is given by $\eps=0$. The equivalence between \eqref{eq:uv} and \eqref{eq:diracscaled} is valid only for $\ep>0$, while \eqref{eq:diracscaled} makes sense for arbitrary $\ep\in\R$.

\end{remark}


\subsection{Solutions of the rescaled problem}
Our goal is to apply the implicit function theorem to prove the existence of a local branch of solutions to \eqref{eq:diracscaled}.

To this aim we rewrite reformulate the problem as follows. Define the map
\[
\cF:\R\times X_r\longrightarrow L^2(\R)\times L^2(\R)\,,
\]
with $X_r$ in \eqref{eq:radialspace}, and acting as
\be
\label{eq:map}
\cF(\ep,u_{\ep},v_{\ep}):=\begin{pmatrix}\displaystyle f'+\frac{f}{\rho}-(\eps\beta_1f^2+2\beta_2g^2)g+g\\[.2cm]
 \displaystyle g' +(2\beta_2\eps^2f^2+\beta_1\eps g^2)f+(2m-\eps)f\end{pmatrix}
\ee

Therefore, the original problem is equivalent to the following
\be
\left\{
\begin{array}{l}
\label{eq:funceq}
\displaystyle (\ep,u_\ep,v_\ep)\in\R\times X_r\\[.2cm]
\displaystyle u_\ep,\,v_\ep\neq0\\[.2cm]
\displaystyle \cF(\ep,u_\ep,v_\ep)=0.
\end{array}
\right. \,\qquad \eps>0\,.
\ee

\begin{remark}\label{rem:onehalf}
In order to simplify the notation, without loss of generality, we take $m=1/2$ and $\beta_{2}=1/2$.
\end{remark}

\begin{proposition}\label{prop:smallepssol}
There exists $\eps_0>0$ such that \eqref{eq:funceq} admits a solution for $\eps\in(-\eps_0,\eps)$.
\end{proposition}
\begin{remark}
The above proposition is equivalent to the main result stated in Theorem \ref{thm:main}. Then we equivalently prove the former.
\end{remark}
\subsubsection{Solutions for $\ep=0$}
\label{sec-nls}
Take $\ep=0$. Looking for non-trivial solutions of \eqref{eq:funceq} in $X_r$ we get
\be\label{eq:eps0}
 \left\{
 \begin{array}{l}
 \displaystyle f'+\frac{f}{\rho}=g^3-g\\[.3cm]
  \displaystyle g'=- f\,.
 \end{array}\right.\,,
\ee
see Remark \ref{rem:onehalf}. Then $(f,g)$ solves the following nonlinear Schr\"odinger equation 
\be
\label{eq:NLSaux} 
\left\{
 \begin{array}{l}
 \displaystyle -g''-\frac{1}{\rho}g'-g^3+g=0\\[.3cm]
  \displaystyle f=-g'\,.
 \end{array}\right.\,.
\ee
Since $\frac{\partial^2}{\partial \rho^2}+\frac{1}{\rho}\frac{\partial}{\partial\rho}$ is the radial part of the two-dimensional Laplacian, one immediately recognizes that $g$ must be a radial solution of the following elliptic equation
\be\label{eq:nls}
-\Delta U-U^3+U=0\,\qquad\mbox{on $\R^2$.}
\ee
It is well known that such equation admits a unique positive radial \emph{ground state solution} $U$, which is smooth and exponentially decaying at infinity \cite[Thm. 8.1.5]{CAZ}. Given such function, we shall consider the solution to \eqref{eq:NLSaux} given by
\be
\label{eq:NLSsol}
(u_0,v_0)=(U,V)\,,\qquad V=-U'\,.
\ee

\subsubsection{Solutions for small $\ep$}
\label{sec-smallep}

In order to prove existence of solutions of \eqref{eq:funceq} for small values of $\ep$ we have to check the assumptions of the implicit function theorem.

It is not hard to verify that the map $\cF$ is of class $C^1$, and then we need to prove the following

\begin{lemma}
\label{prop:linearized}
The differential of $\cF$ with respect to $(u,v)$-variables, $D_{(u,v)}\cF$, evaluated at $(0,u_0,v_0)$ is an isomorphism.
\end{lemma}
The proof of this lemma requires the results that are contained in the following lemmas.

\begin{lemma}
\label{lem-inj}
The operator $D_{(X, Y)}\cF(0,u_0,v_0):X_r\to L^2(\R^2,\R^2)$ is injective.
\end{lemma}
\begin{proof}
We need to prove that $\textrm{ker}\{D_{(u,v)}\cF(0,u_0,v_0)\}$ is trivial. For this reason let us consider the linearization of \eqref{eq:eps0} at $(0,u_0,v_0)$ and $(h,k)\in\ker D_{(u,v)}\cF(0,u_0,v_0)$, so that
\be\label{eq:linearization}
D_{(u,v)}\cF(0,u_0,v_0)[h,k]:=\begin{pmatrix}\displaystyle k'+\frac{k}{\rho}+h-u_0^2h\\[.2cm]
 \displaystyle h'+k\end{pmatrix} =0
\ee
Then $h$ solves 
\[
 -h''-\frac{1}{\rho}h'-u^2_0h+h=0
\]
that is, $h$ lies in the kernel of the linearization of \eqref{eq:nls} at the ground state solution $U$. By know results \cite{NLSspectra}, such kernel is empty and thus $h\equiv0$ and $k\equiv0$, proving the Lemma.
\end{proof}

Now we want to prove that $D_{(u,v)}\cF(0,u_0,v_0)$ is surjective, using the Fredholm alternative \cite[Thm. 6.6]{functional}. Namely, using classical arguments from perturbation theory of linear operators, the claim follows showing that $D_{(u,v)}\cF(0,u_0,v_0)$ is given by the sum of an isomorphism and a compact operator.

By \eqref{eq:linearization}, let
\begin{equation}
\label{eq-dec}
 D_{(u,v)}\cF(0,u_0,v_0)=J+K(u_0),
\end{equation}
where $J,K(u_0):X_r\to L^2(\R^2,\R^2)$ are defined as
\be
\label{eq:J}
J[h,k]:=\left(k'+\frac{k}{\rho}+h,h'+k\right)^T,
\ee
and
\be
\label{eq:K}
K(U)[h,k]:=(u_0h,0)^T\,.
\ee

\begin{lemma}
\label{lem-J}
 The operator $J:X_r\to L^2(\R^2,\R^2)$ is an isomorphism.
\end{lemma}

\begin{proof}
The operator $J$ is clearly continuous, so that we only need to prove injectivity and surjectivity.
 
 \smallskip
 \emph{Step (i): J is injective.} Assume $(h,k)^T\in X_r$ solves $J[h,k]=0$. The argument in the proof of Lemma \ref{lem-inj} gives
\[
-h''+\frac{1}{\rho}h'+h=0\,,
\]
i.e.
\[
-\Delta h+h=0\,,
\]
so that multiplying by $h$ and integrating by parts one immediately sees $h\equiv0$. By \eqref{eq:J} we also get $k\equiv0$.

 \smallskip
 \emph{Step (ii): J is surjective.} Let $a,b\in L^2(\R^2)$. We want to prove that there exists $(h,k)^T\in X_r$ such that
\be
\label{eq:surj}
J[h,k]=(a,b)^T,
\ee
that is, such that
\[
 \left\{\begin{array}{l}
  \displaystyle k'+\frac{k}{\rho}+h= a\\[.2cm]
  \displaystyle h'+k= b
 \end{array}\right.,
 \]
Assuming that $b\equiv0$, arguing as in \emph{Step (i)}, we have to find a weak solution $h_1\in H^1(\R^2)$ of
\be
\label{eq:bzero2}
 -\Delta h_1+h_1=a\,.
\ee
 The existence of such function is an immediate application of the Lax-Milgram Lemma \cite[Cor. 5.8]{functional}. Then taking $k_{1}=-h_1'$, the pair $(h_1,k_1)^T\in X_r$ solves \eqref{eq:surj} with $b\equiv0$.
The same argument, exchaging $k$ and $h$, allows to find a solution $(h_2,k_2)^T\in X_r$ of \eqref{eq:surj} with $a\equiv0$. By linearity of $J$ we thus get the claim.
\end{proof}

\begin{lemma}
\label{lem-K}
 The operator $K(u_0):X_r\to L^2(\R^2,\R^2)$ is compact.
\end{lemma}

\begin{proof}
Let $\big((h_n,k_n)^T\big)_{n}\subset X_r$ be a bounded sequence.

Note that, up to subsequences,
\be
\label{eq:tail}
(h_n,k_n)\longrightarrow (h,k)\qquad \text{in}\quad L_{loc}^2(\R^2,\R^2).
\ee
On the other hand, since the soliton $u_0$ tends to zero at infinity, for all $\eta>0$ there exists $M_\eta>0$ such that $u_0(x)<\eta$ if $|x|>M_\eta$ . Thus,
\begin{multline*}
\|K(u_0)[h_n,k_n]-K(u_0)[h,k]\|_{L^2(\R^2)}=\|u_0^{2}(h_n-h)\|_{L^2(\R^2\setminus B_{M_\eta})}\\[.2cm]
+\|u_0^{2}(h_n-h)\|_{L^2(B_{M_\eta})}\leq C\eta^{2}+o(1),\qquad\text{as}\quad n\to\infty.
\end{multline*}
where $B_{M_\eta}:=\{\vert x\vert\leq B_{M_\eta}\}\subset\R^2$. Then
\[
\lim_{n\to\infty}\|K(u_0)[h_n,k_n]-K(u_0)[h,k]\|_{L^2(\R^2)}\leq C\eta^{2},\qquad\forall \eta>0,
\]
so that the statement follows.
\end{proof}

Now, we can combine all the previous results to prove Proposition \ref{prop:linearized}.

\begin{proof}[Proof of Proposition \ref{prop:linearized}]
 Notice that $D_{(X, Y)}\cF(0,u_0,v_0)$ is clearly continuous, and it is injective by Lemma \ref{lem-inj}. On the other hand, see \eqref{eq-dec}, Lemmata \ref{lem-J} and \ref{lem-K} show that 
 \[
 D_{(X, Y)}\cF(0,u_0,v_0)=J+K
 \] is the sum of an isomorphism and of a compact operator. Then  
 \[
 D_{(X, Y)}\cF(0,u_0,v_0)=J+K=J(I+J^{-1}K)\,.
 \]
 Since $J$ is an isomorphism, the map $(I+J^{-1}K)$, which is of the form identity plus compact, is also injective. Then the claim follows by Fredholm's alternative \cite[Thm. 6.6]{functional}.\end{proof}
 The proof of Theorem \ref{thm:main} immediately follows, as we can now prove Proposition \ref{prop:smallepssol}.
 \begin{proof}[Proof of Proposition \ref{prop:smallepssol}]
There holds $\cF(0,u_0,v_0)=0$ and by Proposition \ref{prop:linearized}, the differential $D_{(u,v)}\cF(0,u_0,v_0)$ is an isomorphism. Then the claim follows by the implicit function theorem.
 \end{proof}

\bibliographystyle{siam}
\bibliography{Mult}

\end{document}